%% file: ch.tex
\newcommand{\home}{}
\documentclass[12pt]{article}
\usepackage{\home razb}
\input{\home razb}

\usepackage{epic,eepic}
\usepackage{amsfonts}
\usepackage{amsmath}

\newcommand{\edge}[2]{\langle #1,#2 \rangle}
\newcommand{\od}{d^+}

\title{On the Caccetta-H\"aggkvist Conjecture with Forbidden Subgraphs}
\author{Alexander A. Razborov\thanks{University of Chicago, {\tt razborov@cs.uchicago.edu}. Part of this work was done while the
author was at
Steklov Mathematical Institute, supported by the Russian Foundation
for Basic Research, and at Toyota Technological Institute, Chicago.}}
\begin{document}
\maketitle

\begin{abstract}
The Caccetta-H\"aggkvist conjecture made in 1978 asserts that every oriented graph on $n$ vertices without oriented cycles of
length $\leq\ell$ must contain a vertex of outdegree at most $\frac{n-1}{\ell}$. It has a rather elaborate set of
(conjectured) extremal configurations.

In this paper we consider the case $\ell=3$ that received quite a significant attention in the literature. We
identify three oriented graphs on four vertices each that are missing as an induced subgraph in all known extremal examples
and prove the Caccetta-H\"aggkvist conjecture for oriented graphs missing as induced subgraphs any of these oriented graphs, along with $\vec C_3$.
Using a standard trick, we can also lift the restriction of being induced, although this makes graphs in our list
slightly more complicated.
\end{abstract}

\section{Introduction} \label{intro}

One prominent way to attack a difficult problem in extremal combinatorics is by better understanding the nature of its
(conjectured) extremal configurations. What one would hope for is to find some property $P$, as ``natural'' as possible
that is shared by all known extremal configurations, and then solve the extremal problem in question for all
configurations possessing this property $P$. Arguably but conceivably, this may shed some light on the nature of
difficulties surrounding the problem in question and perhaps even open up a possibility to solve the problem by
gradually lifting constraints defining the property $P$. For the famous Turan's (3,4)-problem this approach was
recently undertaken by the author in \cite{turan,fdf}; another good example of this sort is the recent solution of the
local Sidorenko conjecture by Lov\'asz \cite{Lov3}.

\medskip
In this paper we address along these lines another major open problem in the area, Caccetta-H\"aggkvist conjecture, that is nearly
as famous as those mentioned above. Recall that we are given an oriented graph\footnote{ That is, a digraph without loops, parallel or
anti-parallel edges. By analogy with the abbreviation ``digraph'',  in this paper oriented graphs  will be often called {\em orgraphs}.} $G$ on $n$ vertices that does not contain (oriented) cycles of length $\leq\ell$ or, in other words,
has girth $\geq\ell+1$. Behzad, Chartrand and Wall \cite{BCWa} asked the following question: if $G$ is additionally
known to be bi-regular, how large can be its degree? They conjectured that the answer is $\lfloor
\frac{n-1}{\ell}\rfloor$ and presented a simple construction attaining this bound. Eight years later, Caccetta and
H\"aggkvist \cite{CaH} proposed to lift in this conjecture the restriction of bi-regularity and, moreover, restrict
attention to minimal {\em outdegree} only. In other words, they asked if every orgraph without oriented
cycles of length $\leq\ell$ must contain a vertex of out-degree $\leq \frac{n-1}{\ell}$, and it is this question that
became known as the {\em Caccetta-H\"aggkvist conjecture}. It turned out to be notoriously difficult, too.

The case of higher values of $\ell$ was studied in \cite{ChS2,Ham,HoR,Nish,Shen,Shen2}.

In this paper we concentrate on the case $\ell=3$, as much of the previous work in this area did. Let $c$ be the
minimal constant for which the asymptotic upper bound $(c+o(1))n$ on the minimal outdegree in $\vec C_3$-free orgraphs
holds. Caccetta and H\"aggkvist themselves proved in \cite{CaH} that $c\leq \frac{3-\sqrt 5}{2}\approx 0.382$. This was
improved in the series of papers \cite{Bon,Shen3,HHKo} to the current record of $c\leq 0.3465n$ \cite{HKN}.

As we already noticed, the first example of an orgraph $G$ on $n$ vertices without copies of $\vec C_3$ and minimal degree
$\lfloor \frac{n-1}{3} \rfloor$ was given in the paper \cite{BCWa}. It is quite simple: assuming that $n=3h+1$ for some integer $h$, we let ${\Bbb Z}_{3h+1}$ be the set of vertices, and we connect $i$ to $j$ if and only if $j-i\leq h\mod (3h+1)$. But this example is not unique:
Bondy \cite[Proposition 1]{Bon} observed that the class of orgraphs with the above properties is closed under lexicographic product, which leads to many more non-isomorphic extremal examples for the Caccetta-H\"aggkvist conjecture.

\bigskip
The first (minor) contribution of our paper (Section \ref{description}) consists in a slight generalization of Bondy's construction
which results in what we believe to be the complete set of currently known
conjectured extremal configurations.

All these examples (for the case $\ell=3$) have the property that they are missing (as induced subgraphs) the three
orgraphs shown on Figure \ref{forbidden} (cf. \cite[Proposition 2]{Bon}).
\begin{figure}[tb]
\begin{center}
\input{forbidden.eepic}
\caption{\label{forbidden} Forbidden orgraphs.}
\end{center}
\end{figure}
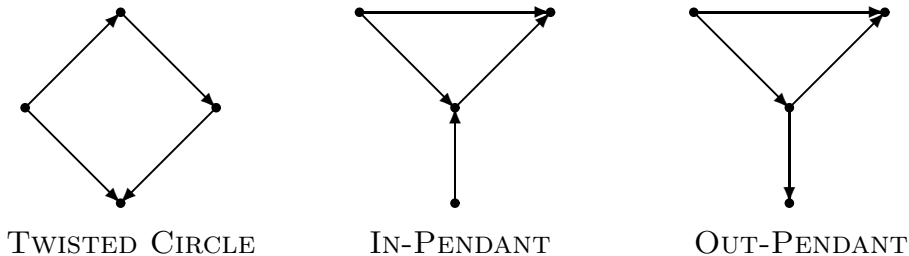
As our main result, we prove the CH-conjecture (for $\ell=3$) for any $\vec C_3$-free orgraph with this additional
property (Theorem \ref{main}). While this is the first result of this kind pertaining to {\em all} known extremal configurations, we would like to mention some previous (unpublished) work regarding forbidden orgraphs on 3 vertices that are missing in the original ``cyclic'' configuration by Behzad, Chartrand and Wall. On Figure \ref{stars}, these are represented by $I_3, \vec K_{1,2}$ and $\vec K_{2,1}$.
\begin{figure}[tb]
\begin{center}
\input{stars.eepic}
\caption{\label{stars} Some orgraphs on 3 vertices.}
\end{center}
\end{figure}
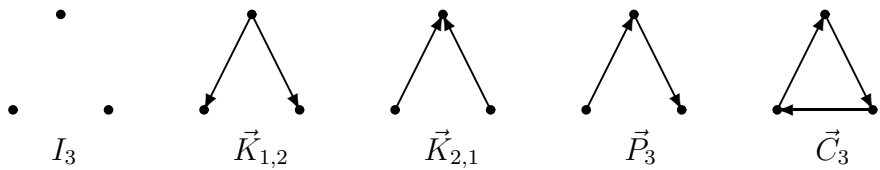

The CH-conjecture (as always, for $\ell=3$) is an easy exercise for orgraphs missing $\vec C_3$ and $\vec K_{1,2}$ as induced subgraphs. Under the additional assumption of out-regularity, Chudnovsky and Seymour \cite{ChSe} did the case when $\vec C_3$ and $I_3$ are missing;  to the best of our knowledge, the question is still open without the restriction of out-regularity. Seymour \cite{Sey}  proved the CH-conjecture for orgraphs missing $\vec C_3$ and $\vec K_{2,1}$ (which is substantially more difficult than the dual case of $\vec C_3$ and $\vec K_{1,2}$).

Potential usefulness of Theorem \ref{main} (at least, of the sort we can think of) as stated above is undermined by the fact
that it involves the notion of an {\em induced} subgraph. We include a very simple observation (Theorem
\ref{supplementary}) showing that this restriction can be removed at the expense of the forbidden family becoming
slightly more complicated (see Figure \ref{not_induced}).

The proof of Theorem \ref{main} was found mostly in the framework of flag algebras \cite{flag}. But the final
calculation (see the crucial Claim \ref{crucial}) does not use multiplicative structure (and, in particular, does not
use Cauchy-Schwarz inequality). This makes working in the limit framework unnecessary, and in this paper we adopt a compromise
approach. Namely, we exclusively work with finite objects but still use basic elements of the whole apparatus of flag
algebras that in our case boils down to two conventions:
\begin{itemize}
\item systematic and consistent notation for various sets based upon types and flags

\item systematic measurement of all necessary quantities in terms of their ``densities'' rather than absolute size.
\end{itemize}

We would like to note that even with this compromise approach lower-order terms do begin to accumulate in the proof of
Claim \ref{crucial}, and we can not simply dismiss them due to the inductive nature of the argument. Fortunately, the
proof is over before they become a real nuisance.

\section{Extremal configurations} \label{description}

We let $[n]\df \{1,2,\ldots,n\}$.

An {\em oriented graph}, or an {\em orgraph}, is a directed graph without loops and such that every pair of vertices is
connected by at most one edge, regardless of direction. $V(\Gamma)$ is the set of vertices of an orgraph $\Gamma$, and
$E(\Gamma)$ is its set of edges. For an edge $\edge vw$, $w$ is its {\em head} and $v$ is its {\em tail}.
$\od_\Gamma(v)$ is the out-degree of the vertex $v$. Vertices $v$ and $w$ are {\em independent} in $\Gamma$ if
neither $\edge vw$ nor $\edge wv$ is an edge. For an orgraph $\Gamma$ and $V\subseteq V(\Gamma)$, $\Gamma|_V$ is the
orgraph induced by $V$.

$\vec C_3$ is the cycle on 3 vertices, and an orgraph $\Gamma$ is {\em $\vec C_3$-free} if it does not contain copies
of $\vec C_3$.

\begin{conjecture}[Caccetta-H\"aggkvist Conjecture] \label{ch}
Any $\vec C_3$-free orgraph $\Gamma$ on $n$ vertices contains a vertex $v$ with $\od_\Gamma(v)\leq
\frac{n-1}{3}$.
\end{conjecture}
We will sometimes refer to this as to ``the CH-conjecture''.

In this section we review what we believe to be the complete list of known configurations attaining this value; as we
noted in Introduction, this is a rather straightforward generalization of the examples found in \cite{BCWa,Bon}. There
are two legitimate frameworks in which this question can be addressed; one of them is exact (i.e., describing finite
orgraphs precisely matching the bound in Conjecture \ref{ch} precisely). And another is asymptotic: it can be best described in
terms of (or)graphons \cite{LoSz} or flag algebras \cite{flag}, but intuitively it corresponds to ``convergent''
sequences of orgraphs $\{\Gamma_m\}$ with $\min_v \od_{\Gamma_m}(v)\geq \of{\frac 13-o(1)}|V(\Gamma_m)|$. We treat them
simultaneously.

Let $S^1$ be the unit circle, and define the (infinite) orgraph $\Gamma_0$ with $V(\Gamma_0)\df S^1$ and
$E(\Gamma_0)\df \set{\edge xy}{y-x<1/3 \mod 1}$. Note that $\Gamma_0$ is $\vec C_3$-free. Let $\Omega\df (S^1)^\infty$
be the infinite-dimensional torus. We let $\Gamma_{\text{CH}}$ be the orgraph with $V(\Gamma_{\text{CH}})=\Omega$ that
is the lexicographic product of countably many copies of $\Gamma_0$. In other words, for any two different vertices
${\sf x}=(x_1,x_2,\ldots,x_n,\ldots),\ {\sf y}=(y_1,\ldots,y_n\ldots)\in \Omega$ we choose the minimal $d$ for which
$x_d\neq y_d$ and let $\edge{\sf x}{\sf y}\in E(\Gamma_{\text{CH}})$ if and only if $\edge{x_d}{y_d}\in E(\Gamma_0)$.

$\Omega$ is a topological space (under product topology), and therefore every probability measure $\mu$ on its Borel
subsets gives rise to an oriented graphon \cite{LoSz}, as well as to a homomorphism $\phi\in\Hom^+(\scr
A^0[T_{\text{CH}}], {\Bbb R})$ \cite{flag}, where $T_{\text{CH}}$ is the theory of $\vec C_3$-free orgraphs. We now describe
those measures $\mu$ that lead to asymptotically extremal examples for Conjecture \ref{ch}.

Fix a probability measure $\mu$ on Borel subsets of $\Omega$. Every finite string $(a^1,\ldots,a^d)\in (S^1)^d$ defines
the canonical closed set $\Omega_a=\set{{\sf x}\in\Omega}{x_1=a_1,\ldots,x_d=a_d}$. Whenever $\mu(\Omega_a)>0$, we
have the conditional measure $\mu_a$ on $\Omega_a$ ($\mu_a(X)\df \frac{\mu(X)}{\mu(\Omega_a)},\ X\subseteq
\Omega_a$) and then the pushforward measure $\widehat\mu_a$ on $S^1$ defined by projecting $\Omega_a$ onto the $(d+1)$st
coordinate. Let us call the measure $\mu$ {\em extremal} if for every prefix $a$ for which $\mu(\Omega_a)>0$, this measure $\widehat\mu_a$ has one of
the following two forms:
\begin{itemize}
\item uniform (Lebesgue) measure on $S^1$;

\item uniform discrete measure on the set $\left\{ \frac 0{3h+1}, \frac 1{3h+1}, \ldots, \frac{3h}{3h+1}\right\}$
    for some integer $h\geq 1$.
\end{itemize}

A combinatorial way to visualize an extremal measure $\mu$ is by a locally finite rooted tree in which every non-leaf node has
outdegree $(3h+1)$ for some $h$; the first case (of Lebesgue measure) corresponds to a leaf.

\begin{claim}
For any extremal measure $\mu$ on $\Omega$ with the above property and for any $\sf x\in\Omega$,
$$\mu(\set{\sf y\in\Omega}{\edge
xy\in E(\Gamma_{\text{CH}})})=1/3.$$
\end{claim}
\begin{proof}
$\set{\sf y\in\Omega}{\edge{\sf x}{\sf y}\in E(\Gamma_{\text{CH}})}$ splits as the disjoint union $\bigcup_{d=1}^\infty
V_d$, where $V_d$ is the set of all $\sf y$ such that $x_1=y_1,\ldots,x_{d-1}=y_{d-1}$ and $y_d -x_d \mod 1 \in (0,1/3)$.
Our restriction on the measures $\widehat\mu_a$ implies that $\mu(V_d)=\frac 13\of{\mu(\Omega_{x_1,\ldots,x_{d-1}}) -
\mu(\Omega_{x_1,\ldots,x_d})}$. Summing over all $d$ and noting that $\mu(\Omega_{x_1,\ldots,x_d})\leq 4^{-d}$ (and
hence $\lim_{d\to\infty} \mu(\Omega_{x_1,\ldots,x_d})=0$) gives the result.
\end{proof}

This collection of oriented graphons describes what we believe to be the complete set of all known extremal
configurations for Conjecture \ref{ch} (more precisely, in the terminology of \cite[\S 4.1]{flag}, the set of all
homomorphisms $\phi\in \Hom(\scr A^0[T_{\text{CH}}], {\Bbb R})$ with $\delta_\alpha(\phi)=1/3$). If we additionally
require the set $\set{a}{\mu(\Omega_a)>0}$ to be finite, we arrive at (again, to the best of our knowledge) the set of
all known finite but in general {\em weighted} (conjecturally) extremal orgraphs. Vertices correspond to leaves of the representing tree,
and if the product of degrees is the same along all terminal paths, then the measure on leaves becomes uniform, and
this gives us (apparently) all known extremal configurations that are {\em ordinary} (unweighted) orgraphs.

One obvious way to ensure the last uniformity property is by requiring that the tree is balanced and all outdegrees are the same
on each level. But there are more sophisticated ways to arrive at a tree with the required property. For example, some
vertices on the first level (we place the root onto level zero) may have $(3g+1)(3h+1)$ leaves as their children, while others
may branch to  a balanced tree of depth 2 with outdegrees $(3g+1)$ on the first level and $(3h+1)$ on the second (and
yet another subtrees may have the same form but with the outdegrees on the first and second level exchanged). These in
particular lead to extremal examples that do {\em not} possess a vertex-transitive group of automorphisms. Nonetheless,
all these examples are bi-regular and in particular are also good for the original question asked in \cite{BCWa}.

\bigskip
Altogether, there are six $\vec C_3$-free orgraphs on four vertices missing in $\Gamma_{\text{CH}}$ as an induced
subgraph \cite[Proposition 2]{Bon}. Of these, we are interested only in the three shown on Figure \ref{forbidden}, and let us first check that they are indeed missing.
\begin{claim}
None of the three orgraphs on Figure \ref{forbidden} can be realized as an induced subgraph of $\Gamma_{\text{CH}}$. \end{claim}
\begin{proof}
Let ${\sf x}^i\in\Omega\ (i=1..4)$ be four different vertices, and let $d$ be the first integer for which there exist
$1\leq i<j\leq 4$ with $x^i_d\neq x^j_d$. The projection onto the $d$th coordinate defines an equivalence relation $\approx$ on $[4]$ with more than one class and such that if $i\approx j\not\approx i'\approx j'$ then $\edge{\sf x^i}{\sf x^{i'}}\in E(\Gamma_{\text{CH}})$ iff $\edge{\sf x^j}{\sf x^{j'}}\in E(\Gamma_{\text{CH}})$. An easy inspection shows that no non-trivial equivalence relation with these properties exists for any of the orgraphs on Figure \ref{forbidden}. Therefore, in fact all $x^i_d\ (i\in [4])$ are pairwise different, and we actually have an embedding into $\Gamma_0$. But $\Gamma_0$ does not contain induced copies of $\vec K_{1,2}, \vec K_{2,1}$ (see Figure \ref{stars}) as an induced subgraph,while every orgraph on Figure \ref{forbidden} contains one of those. Contradiction.
\end{proof}

\section{Main results}

The main result of this paper is the following.

\begin{theorem} \label{main}
Let $\Gamma$ be an orgraph on $n$ vertices that does not contain either $\vec C_3$ or any of the three orgraphs on Figure \ref{forbidden} as an induced subgraph. Then $\Gamma$ contains a vertex $v$ with $\od_\Gamma(v)\leq \frac{n-1}3$.
\end{theorem}

Before we begin the proof of Theorem \ref{main}, let us show how to drop the restriction of being induced.
\begin{theorem} \label{supplementary}
Assume that the CH-conjecture holds for all $\vec C_3$-free orgraphs containing at least one of the orgraphs\footnote{Encircled on this figure are those vertices that are new w.r.t. Figure \ref{forbidden}.} on Figure \ref{not_induced} as a (not necessarily induced) subgraph.
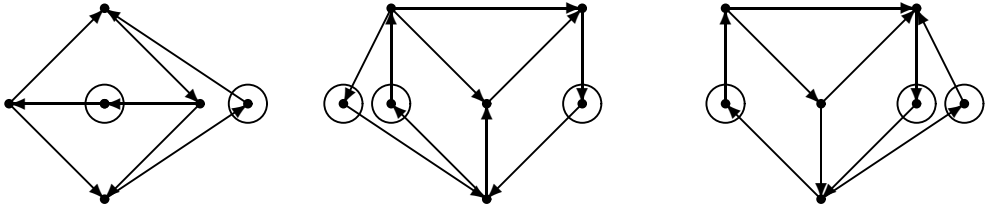
\begin{figure}[tb]
\begin{center}
\input{not_induced.eepic}
\caption{\label{not_induced} Another set of forbidden orgraphs}
\end{center}
\end{figure}
Then the CH-conjecture holds for all $\vec C_3$-free orgraphs.
\end{theorem}
\begin{proof}
Let us describe first how this list of orgraphs was generated from Figure \ref{forbidden}. For every orgraph $\Gamma$ on that figure, we took all ordered pairs of independent vertices $v,w$ such that there is no vertex $x$ with $\edge vx, \edge xw\in E(\Gamma)$. And then for every such pair we added one new vertex $x$ with precisely these edges. On Figure \ref{not_induced}, these auxiliary vertices are encircled.

Assume now that the CH-conjecture holds for all graphs that contain at least one orgraph on Figure \ref{not_induced}. Let $\Gamma$ be an arbitrary $\vec C_3$-free orgraph; we want to show the existence of a vertex $v$ with $\od_\Gamma(v)\leq \frac{n-1}{3}$. W.l.o.g. we may assume that adding any new edges to $\Gamma$ destroys $\vec C_3$-freeness, or, in other words, that every pair $(v,w)$ of independent vertices appears as a diagonal of a copy of $\vec C_4$ in $\Gamma$.

If $\Gamma$ does not have {\em induced} copies of the three orgraphs shown on Figure \ref{forbidden}, we are done by Theorem \ref{main}.

Otherwise, our construction and the remark above imply that $\Gamma$ must contain as a (not necessarily induced) subgraph one of the three orgraphs on Figure \ref{not_induced}, except that some of the encircled vertices can be identical. It is, however, easy to see by inspection that identifying any two of them leads either to anti-parallel edges or a copy of $\vec C_3$, except for the pair of outer-most vertices on the second or the third orgraph. But it is easy to see that the result of this identification will contain the first orgraph on Figure \ref{not_induced} and thus does not need a special treatment.

We have shown that $\Gamma$ must contain one of the three orgraphs on Figure \ref{not_induced}, and therefore the required vertex $v$ exists by our assumption.
\end{proof}

We hope that this piece of information about the structure of hypothetical counterexamples to the CH-conjecture that separates them from all known extremal configurations, may turn out helpful, presumably in combination with inductive arguments of the kind that have been already extensively used in previous research on this problem.

\bigskip
The rest of the paper is entirely devoted to the proof of Theorem \ref{main}. Arguing by induction on the number of vertices, we fix a finite orgraph $\Gamma$ that does not contain either $\vec C_3$ or induced copies of the three orgraphs on Figure \ref{forbidden} and such that the $CH$-conjecture holds for all its proper {\em induced} subgraphs. Our goal is to prove it for $\Gamma$.

\section{Flag Algebras}

As indicated in Introduction, in this paper we use only a tiny fragment of the whole theory, in the amount of the first four pages of \cite[\S 2.1]{flag}.

A {\em type} is a $\vec C_3$-free orgraph $\sigma$ with $V(\sigma)=[k]$ for some non-negative integer $k$ called the {\em size} of $\sigma$. A {\em $\sigma$}-flag is a pair $F=(\Gamma,\theta)$, where $\Gamma$ is a $\vec C_3$-free orgraph and $\theta\function{\sigma}{\Gamma}$ is an induced embedding. Thus, from the combinatorial point of view, a type is just a (totally) labeled orgraph, and a flag is a partially labeled one. Vertices from $V(\Gamma)\setminus \im(\theta)$ will be sometimes called {\em free}.  If $F=(\Gamma,\theta)$ is a $\sigma$-flag and $V\subseteq V(\Gamma)$ contains $\im(\theta)$, then the sub-flag $(\Gamma|_V,\theta)$ will be also denoted by $F|_V$.

A {\em flag embedding} $\alpha\function{F}{F'}$, where $F=(\Gamma,\theta)$ and $F'=(\Gamma',\theta')$ are $\sigma$-flags for the same type $\sigma$ is an induced embedding of orgraphs $\alpha\function{\Gamma}{\Gamma'}$ such that $\theta'=\alpha\theta$ (i.e., ``label-preserving''). $F$ and $F'$ are {\em isomorphic} (denoted by $F\approx F'$) if there is a one-to-one flag embedding $\alpha\function{F}{F'}$.  $\scr F^\sigma_\ell$ is the set of all $\sigma$-flags on $\ell$ vertices.

If $F\in \scr F^\sigma_\ell$ and $F'=(\Gamma,\theta)\in \scr F^\sigma_L$ with $L\geq\ell$, the key quantity in the whole theory is the density $p(F,F')$ of induced copies of $F$ in $F'$ defined as follows. We choose in $V(\Gamma)$ uniformly and at random a subset $\rn V$ of cardinality $\ell$ subject to the condition $\rn V\supseteq \im(\theta)$, and let $p(F,F')$ denote the probability of the event $F'|_{\rn V}\approx F$. In almost all calculations used in this paper, $\ell=k+1$ and hence $\rn V$ can be identified with a single vertex $\rn x$ chosen uniformly at random from $V(\Gamma)\setminus \im(\theta)$.

From now on, we fix an orgraph $\Gamma$ that does not contain either $\vec C_3$ or induced copies of the three orgraphs on Figure \ref{forbidden} and such that the $CH$-conjecture holds for all its proper induced subgraphs. If pairwise different vertices $v_1,\ldots,v_k\in V(\Gamma)$ induce a copy of a type $\sigma$ of size $k$, then, letting $\theta\function{[k]}{V(\Gamma)}$ be the corresponding embedding defined by $\theta(i)=v_i$, $(\Gamma,\theta)$ becomes a $\sigma$-flag, and for another flag $F$ (typically, fixed and very small) we introduce the abbreviation
\begin{equation} \label{relative}
F(v_1,\ldots,v_k)\df p(F,(\Gamma,\theta)).
\end{equation}

\medskip
Next, we list concrete types and flags needed for our purposes.

0 and 1 are the unique type of sizes 0 and 1, respectively. $A$ is the type of size 2 with $E(A)\df \{\edge 12\}$, and $N$ is the type of size 2 without any edges. $P$ is the type of size 3 with $E(P)\df \{\edge 12, \edge 23\}$.

$\alpha\in \scr F^1_2$ is a directed edge in which the tail vertex is labeled by 1; our final goal is to find a vertex $v$ with $\alpha(v)\leq 1/3$. For a type $\sigma$ of size $k$, let $O^\sigma\in\scr F^\sigma_{k+1}$ be the flag in which the only free vertex has $k$ incoming edges. Removing from $O^P$ label 3, we will get a flag from $\scr F^A_4$ that we will denote by $\widehat O^A$, see Figure \ref{oflags}.
\begin{figure}[tb]
\begin{center}
\input{oflags.eepic}
\caption{\label{oflags} $O$-flags}
\end{center}
\end{figure}
We will also need a few other miscellaneous flags from $\scr F^A_3, \scr F^N_3$ shown on Figure \ref{misc}.
\begin{figure}[tb]
\begin{center}
\input{misc.eepic}
\caption{\label{misc} Miscellaneous flags}
\end{center}
\end{figure}
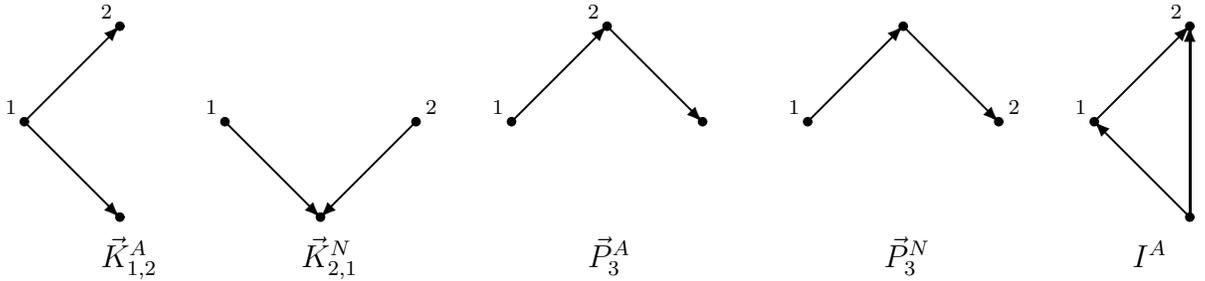

\section{Proof of Theorem \ref{main}}

Let us call an edge $\edge vw$ {\em critical} if $O^A(v,w)$ takes the minimal possible value over all edges going out of $v$. Combinatorially, this means that we are looking at the set $A(v)$ of all out-neighbors of $v$ and pick $w\in A(v)$ to have the smallest possible outdegree in $\Gamma|_{A(v)}$.

\begin{claim} \label{ohata}
For any critical edge $\edge vw$, $\widehat O^A(v,w)=0$.
\end{claim}
\begin{proof}
Assume the contrary, that is $\Gamma$ contains a pair of other vertices $x\neq y$ such that $\{\edge vx,\edge wx, \edge wy, \edge yx\}\in E(\Gamma)$ while $v$ and $y$ are independent. We are going to prove that
\begin{equation} \label{increment}
O^A(v,x) < O^A(v,w),
\end{equation}
and this will contradict the assumption that $\edge vw$ is critical.

For that, let us consider an arbitrary vertex $z$ contributing to $O^A(v,x)$ (that is, such that $\edge vz, \edge xz\in E(\Gamma)$). Since $v,x,y,z$ do not span an In-Pendant (see Figure \ref{forbidden}), $y$ and $z$ may not be independent and thus $\edge yz\in E(\Gamma)$ ($\edge zy$ would have created a copy of $\vec C_3$). But now since $v,w,y,z$ do not span a Twisted Circle, $w$ and $z$ can not be independent and thus $\edge wz\in E(\Gamma)$. Which means that $z$ contributes to $O^A(v,w)$ as well.

Finally, let us note that $x$ itself contributes to $O^A(v,w)$. This completes the proof of \refeq{increment} and gives the desired contradiction with the criticality of $\edge vw$.
\end{proof}

In what follows, we argue by contradiction, i.e. we assume that $\alpha(v)>\frac 13$ for all $v\in V(\Gamma)$.

\begin{claim} \label{p3a}
For any critical edge $\edge vw$, $\vec P_3^A(v,w)>0$.
\end{claim}
\begin{proof}
Note first that $\alpha(w)=\frac{n-2}{n-1}\of{O^A(v,w)+\vec P_3^A(v,w)}$. Next, we can apply the inductive assumption to the set of all out-neighbors of $v$. Since $w$ was chosen to have the minimal outdegree in $\Gamma|_{A(v)}$, $O^A(v,w)\leq \frac 13\alpha(v)\leq\frac 13$. The claim follows immediately since $\alpha(w)> 1/3$ by the assumption we have just made.
\end{proof}

Now we study critical paths of length 2.

\begin{claim} \label{p3}
If $\edge uv$ and $\edge vw$ are critical edges then $u$ and $w$ are independent.
\end{claim}
\begin{proof}
Assume the contrary, that is $\edge uw\in E(\Gamma)$. By Claim \ref{p3a}, there exists $x$ such that $\edge vx\in E(\Gamma)$ while $u$ and $x$ are independent. Since $(u,v,w,x)$ do not induce an Out-Pendant, $w$ and $x$ may not be independent, and the edge $\edge xw$ is ruled out by Claim \ref{ohata}. Therefore, $\edge wx\in E(\Gamma)$.

And now we use the criticality of $\edge vw$, and our goal, like in the proof of Claim \ref{ohata}, is to arrive at a contradiction by establishing \refeq{increment}. We again choose $z$ with $\edge vz, \edge xz\in E(\Gamma)$. The edge $\edge uz$ is again ruled out by Claim \ref{ohata} (applied to $\{u,v,x,z\}$), therefore $u$ and $z$ must be independent.

And now $w,z$ may not be independent (since otherwise $\{u,v,w,z\}$ would have formed an Out-Pendant). Therefore, $\edge wz\in E(\Gamma)$, and the rest of the proof is the same as in the proof of Claim \ref{ohata}.
\end{proof}

The flag $\vec K_{2,1}^N$ is shown on Figure \ref{misc}.

\begin{claim} \label{nok21}
If $\edge uv$ and $\edge vw$ are critical edges then $\vec K^N_{2,1}(u,w)=0$.
\end{claim}
\begin{proof}
Assume the contrary, and let $x$ be any vertex such that $\edge ux, \edge wx\in E(\Gamma)$. Then $v$ and $x$ can not be independent (as it would have created a Twisted Circle), $\edge xv$ can not be an edge since it would have created $\vec C_3$ and $\edge vx$ can not be an edge by Claim \ref{ohata}.
\end{proof}

\begin{claim} \label{induction}
If $\edge uv$ and $\edge vw$ are critical edges then
\begin{equation}\label{required}
    3O^A(u,v) \leq \vec P_3^N(u,w) -\frac 1{n-2}.
\end{equation}
\end{claim}
\begin{proof}
Let $\vec P_3^N(u,w)=\frac h{n-2}$, and let $U\ni v$ be the corresponding set of vertices, $|U|=h$. Applying to $\Gamma|_U$ our inductive assumption, we find a vertex $v^\ast\in U$ that has degree $\leq \frac{h-1}{3}$ in $\Gamma|_U$ (possibly, $v^\ast=v$). We will prove that
\begin{equation} \label{allin}
    O^A(u,v^\ast)\leq \frac{h-1}{3(n-2)}
\end{equation}
from which the claim follows since $O^A(u,v)\leq O^A(u, v^\ast)$ due to the criticality of $\edge uv$.

To prove \refeq{allin}, it suffices to show that every vertex $x$ contributing to $O^A(u, v^\ast)$ in fact belongs to $U$, that is $\edge xw\in E(\Gamma)$. But $w,x$ can not be independent (since otherwise we would get a copy of an Out-Pendant), and the edge $\edge wx$ is ruled out by Claim \ref{nok21} (note that we must apply this claim to the triple $(u,v,w)$, {\em not} $(u, v^\ast,w)$, as we do not know that $\edge u{v^\ast}$ is critical!).
\end{proof}

The following is our crucial claim that is a typical computer-assisted calculation in flag algebras, albeit much simpler than in all previous applications of the method. For an explanation of all new flags appearing in its statement and proof, we refer to Figure \ref{misc}.
\begin{claim} \label{crucial}
If $\edge uv$ and $\edge vw$ are critical edges, then
\begin{equation}\label{crucial:eqn}
\longeq{ &&   \alpha(u)+\alpha(v)+\alpha(w) +(O^A(u,v)+I^A(u,v)+\vec K_{2,1}^A(u,v))\\&&\longeqskip - (O^A(v,w)+I^A(v,w)+\vec K_{2,1}^A(v,w))\leq 1.}
\end{equation}
\end{claim}
\begin{proof}
Subtracting the inequality in Claim \ref{induction} and re-grouping terms, it suffices to prove that
\begin{equation}\label{main_computation}
 \longeq{&&   \alpha(u)+\alpha(v)+\alpha(w) +(I^A(u,v)+\vec K_{2,1}^A(u,v)-2O^A(u,v))\\&&\longeqskip - (O^A(v,w)+I^A(v,w)+\vec K_{2,1}^A(v,w))+\vec P_3^N(u,w)\leq 1+\frac 1{n-2}.}
\end{equation}
Let us pick $\rn x\in V(\Gamma)\setminus \{u,v,w\}$ uniformly at random and let us re-calculate all quantities in the left-hand side of \refeq{main_computation} with respect to that distribution. More precisely, for $\{v_1,\ldots,v_k\}\subset \{u,v,w\}$ we replace the quantity $p(F,(\Gamma,\theta))$ in \eqref{relative} by the probability of the event $(\Gamma,\theta)|_{\{v_1,\ldots,v_k,\rn x\}}\approx F$ (thus, the only difference from the original definition is that now the random variable $\rn x$ is forbidden to take values from $\{u,v,w\}\setminus \{v_1,\ldots,v_k\}$).
Denoting these re-calculated quantities with $\widetilde\alpha(u),\ldots, \widetilde{{\vec P_3}^N}(u,w)$, we claim that
\begin{equation}\label{flag}
\longeq{&&\widetilde\alpha(u)+\widetilde\alpha(v)+\widetilde\alpha(w) +(\widetilde{I^A}(u,v)+\widetilde{\vec K_{2,1}^A}(u,v)-2\widetilde{O^A}(u,v))\\&&\longeqskip - (\widetilde{O^A}(v,w)+\widetilde{I^A}(v,w)+\widetilde{\vec K_{2,1}^A}(v,w))+\widetilde{\vec P_3^N}(u,w)\leq 1.}
\end{equation}
For that we prove that every {\em individual} $x\not\in \{u,v,w\}$ contributes at most 1 to the left-hand side\footnote{The reader familiar with the formalism of flag algebras may notice that we are simply proving the inequality $\sum_{i=1}^3\pi^{P,i}(\alpha) + \pi^{P,[1,2]}(O^A+I^A+\vec K_{2,1}^A) - \pi^{P,[2,3]}(O^A+I^A+\vec K_{2,1}^A) + \pi^{P,[1,3]}(\vec P_3^N)-3\pi^{P,[1,2]}(O^A)\leq 1$.}.

We can assume w.l.o.g. that $x$ contributes to at least two terms 
$$
\widetilde\alpha(u), \widetilde\alpha(v), \widetilde\alpha(w), \widetilde{I^A}(u,v), \widetilde{\vec K_{2,1}^A}(u,v), \widetilde{\vec P_3^N}(u,w),
$$ 
and we have to show that this excessive (over 1) contribution is compensated by the contribution of $x$ to negative terms.

Firstly we note that $x$ may contribute to at most one of the terms $\widetilde{I^A}(u,v), \widetilde{\vec K_{2,1}^A}(u,v), \widetilde{\vec P_3^N}(u,w)$, depending on whether $\edge xu\in E(\Gamma)$, $x,u$ are independent or $\edge ux\in E(\Gamma)$. Thus, we can also assume w.l.o.g. that $x$ contributes to at least one of $\widetilde\alpha(u), \widetilde\alpha(v), \widetilde\alpha(w)$. Which immediately implies that $x$ may not contribute to $\widetilde{I^A}(u,v), \widetilde{\vec K_{2,1}^A}(u,v)$: otherwise we would have had $\edge xv\in E(\Gamma)$, hence $\edge wx\not\in E(\Gamma)$ and $x$ would not have contributed to  $\widetilde\alpha(u) + \widetilde\alpha(v) + \widetilde\alpha(w)$.

Let us now consider the case when $x$ contributes to $\widetilde{\vec P_3^N}(u,w)$, i.e., $\edge ux, \edge xw\in E(\Gamma)$. If $\edge vx\not\in E(\Gamma)$, then $x$ does give an excessive contribution of one (to $\widetilde\alpha(u)$ and $\widetilde{\vec P_3^N}(u,w)$), but it is compensated by its contribution to one of the negative terms $\widetilde{I^A}(v,w), \widetilde{\vec K_{2,1}^A}(v,w)$. If, on the other hand, $\edge vx\in E(\Gamma)$, then the excessive contribution of two is offset by the term $2\widetilde{O^A}(u,v)$.

We are left with the case when $x$ does not contribute to $\widetilde{I^A}(u,v), \widetilde{\vec K_{2,1}^A}(u,v), \widetilde{\vec P_3^N}(u,w)$ at all which, in particular implies that it contributes to at least two terms  $\widetilde\alpha(u), \widetilde\alpha(v), \widetilde\alpha(w)$.

If $x$ contributes to both $\widetilde\alpha(u), \widetilde\alpha(v)$, then this contribution is again offset by $2\widetilde{O^A}(u,v)$. Thus, we can also assume that $x$ contributes to $\widetilde\alpha(w)$ and to {\em precisely} one of the terms $\widetilde\alpha(u), \widetilde\alpha(v)$. If $\edge vx\in E(\Gamma)$ then the excessive conribution is taken care of by the last remaining negative term $\widetilde{O^A}(v,w)$. And the case $\edge ux\in E(\Gamma),\ \edge vx\not\in E(\Gamma)$ is impossible since it gives rise to Twisted Circle.

The proof of \eqref{flag} is complete.

On the other hand, we have
\begin{eqnarray*}
\alpha(u) &=& \frac{n-3}{n-1}\widetilde\alpha(u)+ \frac 1{n-1};\\
\alpha(v) &=& \frac{n-3}{n-1}\widetilde\alpha(v)+ \frac 1{n-1};\\
\alpha(w) &=& \frac{n-3}{n-1}\widetilde\alpha(w);\\
\vec P_3^N(u,w) &=& \frac{n-3}{n-2}\widetilde{\vec P_3^N}(u,w)+ \frac 1{n-2};\\
F(y,z) &=& \frac{n-3}{n-2}\widetilde F(y,z)\ \text{for any other term}\ F(y,z)\ \text{in \refeq{main_computation}, \refeq{flag}.}
\end{eqnarray*}
Multiplying \refeq{flag} by $\frac{n-3}{n-2}$ and adding $\frac 1{n-1}$ to both sides, we get
\begin{equation}\label{intermediate}
    \longeq{&&\frac{n-3}{n-2}\of{\widetilde\alpha(u)+\widetilde\alpha(v)+\widetilde\alpha(w)} +({I^A}(u,v)+\vec K_{2,1}^A(u,v)-2{O^A(u,v)})\\&&\longeqskip - ({O^A}(v,w)+{I^A}(v,w)+{\vec K_{2,1}^A}(v,w))+{\vec P_3^N}(u,w)\leq 1.}
\end{equation}

Also, since $\alpha(u)+\alpha(v)+\alpha(w)> 1$ by our assumption, we have
\begin{eqnarray*}
&&\frac{n-3}{n-2}\of{\widetilde\alpha(u)+\widetilde\alpha(v)+\widetilde\alpha(w)} = \frac{n-1}{n-2}\of{\alpha(u)+\alpha(v)+\alpha(w)}-\frac 2{n-2}\\&&\longeqskip \geq \alpha(u)+\alpha(v)+\alpha(w)-\frac 1{n-2}.
\end{eqnarray*}
Substituting this into \refeq{intermediate} and summing the resulting inequality with \refeq{required} gives us \refeq{main_computation}.
\end{proof}

Now the proof of Theorem \ref{main} is completed by an easy averaging argument. Since for every vertex $v$ there exists at least one critical edge going out of $v$, there exists a cycle $C=(v_1,v_2,\ldots,v_\ell)$ consisting of critical edges. After summing up the inequalities \refeq{crucial:eqn} along this cycle, the terms $O^A(u,v),\ldots,\vec K_{2,1}^A(v,w)$ get canceled and we arrive at
$$
\sum_{i\in {\Bbb Z}_\ell}\alpha(v_i)\leq\ell/3.
$$
Therefore, there exists at least one $i$ with $\alpha(v_i)\leq 1/3$. Theorem \ref{main} is proved.

\bigskip
That would be interesting to improve upon our result by removing some (and preferably all) forbidden subgraphs on Figure \ref{forbidden}. We have tried it for a while, but all three constraints are very essential in our proof, and removing any one of them immediately creates a new level of difficulties that we have not been able to surpass.

\section*{Acknowledgment}

I am grateful to Jan Hladk\'y for several useful remarks. I am also indebted to both anonymous referees for reading the manuscript very carefully that has resulted in several important changes and corrections.

\section*{Added in proof}

Lichiardopol \cite{Lic} has given an affirmative answer to the question asked in Section \ref{intro}: the CH-conjecture holds for oriented graphs with independence number 2 (and without restriction of out-regularity).

\bibliographystyle{alpha}
\bibliography{\home razb}
\end{document}

%% file: razb.tex
\newcommand{\refeq}[1]{{\rm (\ref{#1})}}
\newlength{\longeqmarginwidth}
\newlength{\longeqwidth}
\newlength{\longeqskiplength}
\setlength{\longeqskiplength}{\parindent}
\newcommand{\longeqskip}{\hspace{\longeqskiplength}}
\newcommand{\longeq}[1]{
\settowidth{\longeqmarginwidth}{{\Large \{~\}}~(\theequation)}
\setlength{\longeqwidth}{\textwidth}
\addtolength{\longeqwidth}{-\longeqmarginwidth}
\left. \parbox{\longeqwidth}{\begin{eqnarray*} #1
\end{eqnarray*}} \right\}}


\newcommand{\Vzero}[2]{\mathord{\stackrel{\textstyle\kern1pt\circ}
{\smash V\vbox to6pt{}}}\vphantom{V}_{#1}^{#2}}

\settowidth{\unitlength}{$-$}

\newcommand{\function}[2]{:#1 \longrightarrow #2}
\newcommand{\of}[1]{\left( #1 \right)}

\newcommand{\set}[2]{\left\{\hspace{0.2ex} #1 \left|\: #2
\right. \right\}}

\newcommand{\df}{\stackrel{\rm def}{=}}

\newcommand{\im}{{\rm im}}

\newcommand{\Hom}{{\rm Hom}}

\newcommand{\rn}{\boldsymbol}
\newcommand{\scr}{\mathcal}

\newcounter{operator}

%% file: forbidden.eepic
\setlength{\unitlength}{0.254mm}
\begin{picture}(462,130)(0,-138)
        \special{color rgb 0 0 0}\allinethickness{0.254mm}\special{sh 0.99}\put(10,-60){\ellipse{4}{4}} 
        \special{color rgb 0 0 0}\allinethickness{0.254mm}\special{sh 0.99}\put(60,-10){\ellipse{4}{4}} 
        \special{color rgb 0 0 0}\put(0,-136){\shortstack{\sc Twisted Circle}} 
        \special{color rgb 0 0 0}\allinethickness{0.254mm}\put(60,-10){\vector(1,-1){50}} 
        \special{color rgb 0 0 0}\allinethickness{0.254mm}\put(10,-60){\vector(1,1){50}} 
        \special{color rgb 0 0 0}\allinethickness{0.254mm}\put(10,-60){\vector(1,-1){50}} 
        \special{color rgb 0 0 0}\allinethickness{0.254mm}\put(110,-60){\vector(-1,-1){50}} 
        \special{color rgb 0 0 0}\allinethickness{0.254mm}\special{sh 0.99}\put(110,-60){\ellipse{4}{4}} 
        \special{color rgb 0 0 0}\allinethickness{0.254mm}\special{sh 0.99}\put(60,-110){\ellipse{4}{4}} 
        \special{color rgb 0 0 0}\allinethickness{0.254mm}\special{sh 0.99}\put(185,-10){\ellipse{4}{4}} 
        \special{color rgb 0 0 0}\allinethickness{0.254mm}\special{sh 0.99}\put(235,-60){\ellipse{4}{4}} 
        \special{color rgb 0 0 0}\allinethickness{0.254mm}\put(185,-10){\vector(1,-1){50}} 
        \special{color rgb 0 0 0}\allinethickness{0.254mm}\special{sh 0.99}\put(285,-10){\ellipse{4}{4}} 
        \special{color rgb 0 0 0}\allinethickness{0.254mm}\special{sh 0.99}\put(235,-110){\ellipse{4}{4}} 
        \special{color rgb 0 0 0}\allinethickness{0.254mm}\put(185,-10){\vector(1,0){100}} 
        \special{color rgb 0 0 0}\allinethickness{0.254mm}\put(235,-60){\vector(1,1){50}} 
        \special{color rgb 0 0 0}\allinethickness{0.254mm}\put(235,-110){\vector(0,1){50}} 
        \special{color rgb 0 0 0}\put(190,-136){\shortstack{\sc In-Pendant}} 
        \special{color rgb 0 0 0}\allinethickness{0.254mm}\special{sh 0.99}\put(360,-10){\ellipse{4}{4}} 
        \special{color rgb 0 0 0}\allinethickness{0.254mm}\special{sh 0.99}\put(410,-60){\ellipse{4}{4}} 
        \special{color rgb 0 0 0}\allinethickness{0.254mm}\put(360,-10){\vector(1,-1){50}} 
        \special{color rgb 0 0 0}\allinethickness{0.254mm}\special{sh 0.99}\put(460,-10){\ellipse{4}{4}} 
        \special{color rgb 0 0 0}\allinethickness{0.254mm}\special{sh 0.99}\put(410,-110){\ellipse{4}{4}} 
        \special{color rgb 0 0 0}\allinethickness{0.254mm}\put(360,-10){\vector(1,0){100}} 
        \special{color rgb 0 0 0}\allinethickness{0.254mm}\put(410,-60){\vector(1,1){50}} 
        \special{color rgb 0 0 0}\allinethickness{0.254mm}\put(410,-60){\vector(0,-1){50}} 
        \special{color rgb 0 0 0}\put(360,-136){\shortstack{\sc Out-Pendant}} 
        \special{color rgb 0 0 0} 
\end{picture}

%% file: stars.eepic
\setlength{\unitlength}{0.254mm}
\begin{picture}(488,80)(13,-93)
        \special{color rgb 0 0 0}\allinethickness{0.254mm}\special{sh 0.99}\put(15,-65){\ellipse{4}{4}} 
        \special{color rgb 0 0 0}\allinethickness{0.254mm}\special{sh 0.99}\put(40,-15){\ellipse{4}{4}} 
        \special{color rgb 0 0 0}\allinethickness{0.254mm}\special{sh 0.99}\put(65,-65){\ellipse{4}{4}} 
        \special{color rgb 0 0 0}\allinethickness{0.254mm}\special{sh 0.99}\put(115,-65){\ellipse{4}{4}} 
        \special{color rgb 0 0 0}\allinethickness{0.254mm}\special{sh 0.99}\put(140,-15){\ellipse{4}{4}} 
        \special{color rgb 0 0 0}\allinethickness{0.254mm}\special{sh 0.99}\put(165,-65){\ellipse{4}{4}} 
        \special{color rgb 0 0 0}\put(35,-91){\shortstack{$I_3$}} 
        \special{color rgb 0 0 0}\put(130,-91){\shortstack{$\vec K_{1,2}$}} 
        \special{color rgb 0 0 0}\allinethickness{0.254mm}\put(140,-15){\vector(-1,-2){25}} 
        \special{color rgb 0 0 0}\allinethickness{0.254mm}\put(140,-15){\vector(1,-2){25}} 
        \special{color rgb 0 0 0}\allinethickness{0.254mm}\special{sh 0.99}\put(415,-65){\ellipse{4}{4}} 
        \special{color rgb 0 0 0}\allinethickness{0.254mm}\special{sh 0.99}\put(440,-15){\ellipse{4}{4}} 
        \special{color rgb 0 0 0}\allinethickness{0.254mm}\special{sh 0.99}\put(465,-65){\ellipse{4}{4}} 
        \special{color rgb 0 0 0}\put(435,-91){\shortstack{$\vec C_3$}} 
        \special{color rgb 0 0 0}\allinethickness{0.254mm}\put(440,-15){\vector(1,-2){25}} 
        \special{color rgb 0 0 0}\allinethickness{0.254mm}\put(415,-65){\vector(1,2){25}} 
        \special{color rgb 0 0 0}\allinethickness{0.254mm}\special{sh 0.99}\put(315,-65){\ellipse{4}{4}} 
        \special{color rgb 0 0 0}\allinethickness{0.254mm}\special{sh 0.99}\put(340,-15){\ellipse{4}{4}} 
        \special{color rgb 0 0 0}\allinethickness{0.254mm}\special{sh 0.99}\put(365,-65){\ellipse{4}{4}} 
        \special{color rgb 0 0 0}\put(335,-91){\shortstack{$\vec P_3$}} 
        \special{color rgb 0 0 0}\allinethickness{0.254mm}\put(340,-15){\vector(1,-2){25}} 
        \special{color rgb 0 0 0}\allinethickness{0.254mm}\put(315,-65){\vector(1,2){25}} 
        \special{color rgb 0 0 0}\allinethickness{0.254mm}\put(465,-65){\vector(-1,0){50}} 
        \special{color rgb 0 0 0}\allinethickness{0.254mm}\special{sh 0.99}\put(215,-65){\ellipse{4}{4}} 
        \special{color rgb 0 0 0}\allinethickness{0.254mm}\special{sh 0.99}\put(240,-15){\ellipse{4}{4}} 
        \special{color rgb 0 0 0}\allinethickness{0.254mm}\special{sh 0.99}\put(265,-65){\ellipse{4}{4}} 
        \special{color rgb 0 0 0}\put(230,-91){\shortstack{$\vec K_{2,1}$}} 
        \special{color rgb 0 0 0}\allinethickness{0.254mm}\put(215,-65){\vector(1,2){25}} 
        \special{color rgb 0 0 0}\allinethickness{0.254mm}\put(265,-65){\vector(-1,2){25}} 
        \special{color rgb 0 0 0} 
\end{picture}

%% file: not_induced.eepic
\setlength{\unitlength}{0.254mm}
\begin{picture}(512,104)(8,-112)
        \special{color rgb 0 0 0}\allinethickness{0.254mm}\special{sh 0.99}\put(10,-60){\ellipse{4}{4}} 
        \special{color rgb 0 0 0}\allinethickness{0.254mm}\special{sh 0.99}\put(60,-10){\ellipse{4}{4}} 
        \special{color rgb 0 0 0}\allinethickness{0.254mm}\put(60,-10){\vector(1,-1){50}} 
        \special{color rgb 0 0 0}\allinethickness{0.254mm}\put(10,-60){\vector(1,1){50}} 
        \special{color rgb 0 0 0}\allinethickness{0.254mm}\put(10,-60){\vector(1,-1){50}} 
        \special{color rgb 0 0 0}\allinethickness{0.254mm}\put(110,-60){\vector(-1,-1){50}} 
        \special{color rgb 0 0 0}\allinethickness{0.254mm}\special{sh 0.99}\put(110,-60){\ellipse{4}{4}} 
        \special{color rgb 0 0 0}\allinethickness{0.254mm}\special{sh 0.99}\put(60,-110){\ellipse{4}{4}} 
        \special{color rgb 0 0 0}\allinethickness{0.254mm}\special{sh 0.99}\put(210,-10){\ellipse{4}{4}} 
        \special{color rgb 0 0 0}\allinethickness{0.254mm}\special{sh 0.99}\put(260,-60){\ellipse{4}{4}} 
        \special{color rgb 0 0 0}\allinethickness{0.254mm}\put(210,-10){\vector(1,-1){50}} 
        \special{color rgb 0 0 0}\allinethickness{0.254mm}\special{sh 0.99}\put(310,-10){\ellipse{4}{4}} 
        \special{color rgb 0 0 0}\allinethickness{0.254mm}\special{sh 0.99}\put(260,-110){\ellipse{4}{4}} 
        \special{color rgb 0 0 0}\allinethickness{0.254mm}\put(210,-10){\vector(1,0){100}} 
        \special{color rgb 0 0 0}\allinethickness{0.254mm}\put(260,-60){\vector(1,1){50}} 
        \special{color rgb 0 0 0}\allinethickness{0.254mm}\put(260,-110){\vector(0,1){50}} 
        \special{color rgb 0 0 0}\allinethickness{0.254mm}\special{sh 0.99}\put(60,-60){\ellipse{4}{4}} 
        \special{color rgb 0 0 0}\allinethickness{0.254mm}\put(110,-60){\vector(-1,0){50}} 
        \special{color rgb 0 0 0}\allinethickness{0.254mm}\put(60,-60){\vector(-1,0){50}} 
        \special{color rgb 0 0 0}\allinethickness{0.254mm}\put(60,-110){\vector(3,2){75}} 
        \special{color rgb 0 0 0}\allinethickness{0.254mm}\put(135,-60){\vector(-3,2){75}} 
        \special{color rgb 0 0 0}\allinethickness{0.254mm}\special{sh 0.99}\put(135,-60){\ellipse{4}{4}} 
        \special{color rgb 0 0 0}\allinethickness{0.254mm}\put(60,-60){\ellipse{4}{4}} 
        \special{color rgb 0 0 0}\allinethickness{0.254mm}\special{sh 0.99}\put(310,-60){\ellipse{4}{4}} 
        \special{color rgb 0 0 0}\allinethickness{0.254mm}\special{sh 0.99}\put(210,-60){\ellipse{4}{4}} 
        \special{color rgb 0 0 0}\allinethickness{0.254mm}\special{sh 0.99}\put(185,-60){\ellipse{4}{4}} 
        \special{color rgb 0 0 0}\allinethickness{0.254mm}\put(310,-10){\vector(0,-1){50}} 
        \special{color rgb 0 0 0}\allinethickness{0.254mm}\put(310,-60){\vector(-1,-1){50}} 
        \special{color rgb 0 0 0}\allinethickness{0.254mm}\put(260,-110){\vector(-1,1){50}} 
        \special{color rgb 0 0 0}\allinethickness{0.254mm}\put(210,-60){\vector(0,1){50}} 
        \special{color rgb 0 0 0}\allinethickness{0.254mm}\put(210,-10){\vector(-1,-2){25}} 
        \special{color rgb 0 0 0}\allinethickness{0.254mm}\put(185,-60){\vector(3,-2){75}} 
        \special{color rgb 0 0 0}\allinethickness{0.254mm}\special{sh 0.99}\put(385,-10){\ellipse{4}{4}} 
        \special{color rgb 0 0 0}\allinethickness{0.254mm}\special{sh 0.99}\put(435,-60){\ellipse{4}{4}} 
        \special{color rgb 0 0 0}\allinethickness{0.254mm}\put(385,-10){\vector(1,-1){50}} 
        \special{color rgb 0 0 0}\allinethickness{0.254mm}\special{sh 0.99}\put(485,-10){\ellipse{4}{4}} 
        \special{color rgb 0 0 0}\allinethickness{0.254mm}\special{sh 0.99}\put(435,-110){\ellipse{4}{4}} 
        \special{color rgb 0 0 0}\allinethickness{0.254mm}\put(385,-10){\vector(1,0){100}} 
        \special{color rgb 0 0 0}\allinethickness{0.254mm}\put(435,-60){\vector(1,1){50}} 
        \special{color rgb 0 0 0}\allinethickness{0.254mm}\special{sh 0.99}\put(485,-60){\ellipse{4}{4}} 
        \special{color rgb 0 0 0}\allinethickness{0.254mm}\special{sh 0.99}\put(385,-60){\ellipse{4}{4}} 
        \special{color rgb 0 0 0}\allinethickness{0.254mm}\put(485,-10){\vector(0,-1){50}} 
        \special{color rgb 0 0 0}\allinethickness{0.254mm}\put(485,-60){\vector(-1,-1){50}} 
        \special{color rgb 0 0 0}\allinethickness{0.254mm}\put(435,-110){\vector(-1,1){50}} 
        \special{color rgb 0 0 0}\allinethickness{0.254mm}\put(385,-60){\vector(0,1){50}} 
        \special{color rgb 0 0 0}\allinethickness{0.254mm}\put(435,-60){\vector(0,-1){50}} 
        \special{color rgb 0 0 0}\allinethickness{0.254mm}\special{sh 0.99}\put(510,-60){\ellipse{4}{4}} 
        \special{color rgb 0 0 0}\allinethickness{0.254mm}\put(435,-110){\vector(3,2){75}} 
        \special{color rgb 0 0 0}\allinethickness{0.254mm}\put(510,-60){\vector(-1,2){25}} 
        \special{color rgb 0 0 0}\allinethickness{0.254mm}\put(60,-60){\ellipse{20}{20}} 
        \special{color rgb 0 0 0}\allinethickness{0.254mm}\put(135,-60){\ellipse{20}{20}} 
        \special{color rgb 0 0 0}\allinethickness{0.254mm}\put(185,-60){\ellipse{20}{20}} 
        \special{color rgb 0 0 0}\allinethickness{0.254mm}\put(210,-60){\ellipse{20}{20}} 
        \special{color rgb 0 0 0}\allinethickness{0.254mm}\put(310,-60){\ellipse{20}{20}} 
        \special{color rgb 0 0 0}\allinethickness{0.254mm}\put(385,-60){\ellipse{20}{20}} 
        \special{color rgb 0 0 0}\allinethickness{0.254mm}\put(485,-60){\ellipse{20}{20}} 
        \special{color rgb 0 0 0}\allinethickness{0.254mm}\put(510,-60){\ellipse{20}{20}} 
        \special{color rgb 0 0 0} 
\end{picture}

%% file: oflags.eepic
\setlength{\unitlength}{0.254mm}
\begin{picture}(473,146)(25,-153)
        \special{color rgb 0 0 0}\allinethickness{0.254mm}\special{sh 0.99}\put(35,-75){\ellipse{4}{4}} 
        \special{color rgb 0 0 0}\allinethickness{0.254mm}\special{sh 0.99}\put(85,-25){\ellipse{4}{4}} 
        \special{color rgb 0 0 0}\put(65,-151){\shortstack{$O^A$}} 
        \special{color rgb 0 0 0}\allinethickness{0.254mm}\put(35,-75){\vector(1,1){50}} 
        \special{color rgb 0 0 0}\allinethickness{0.254mm}\put(35,-75){\vector(1,-1){50}} 
        \special{color rgb 0 0 0}\allinethickness{0.254mm}\special{sh 0.99}\put(85,-125){\ellipse{4}{4}} 
        \special{color rgb 0 0 0}\allinethickness{0.254mm}\special{sh 0.99}\put(235,-25){\ellipse{4}{4}} 
        \special{color rgb 0 0 0}\put(225,-151){\shortstack{$O^P$}} 
        \special{color rgb 0 0 0}\allinethickness{0.254mm}\put(235,-25){\vector(1,-1){50}} 
        \special{color rgb 0 0 0}\allinethickness{0.254mm}\put(185,-75){\vector(1,1){50}} 
        \special{color rgb 0 0 0}\allinethickness{0.254mm}\put(185,-75){\vector(1,-1){50}} 
        \special{color rgb 0 0 0}\allinethickness{0.254mm}\put(285,-75){\vector(-1,-1){50}} 
        \special{color rgb 0 0 0}\allinethickness{0.254mm}\special{sh 0.99}\put(285,-75){\ellipse{4}{4}} 
        \special{color rgb 0 0 0}\allinethickness{0.254mm}\special{sh 0.99}\put(235,-125){\ellipse{4}{4}} 
        \special{color rgb 0 0 0}\allinethickness{0.254mm}\put(85,-25){\vector(0,-1){100}} 
        \special{color rgb 0 0 0}\put(25,-71){\shortstack{\scriptsize 1}} 
        \special{color rgb 0 0 0}\put(75,-21){\shortstack{\scriptsize 2}} 
        \special{color rgb 0 0 0}\allinethickness{0.254mm}\special{sh 0.99}\put(185,-75){\ellipse{4}{4}} 
        \special{color rgb 0 0 0}\put(175,-71){\shortstack{\scriptsize 1}} 
        \special{color rgb 0 0 0}\put(225,-21){\shortstack{\scriptsize 2}} 
        \special{color rgb 0 0 0}\put(290,-71){\shortstack{\scriptsize 3}} 
        \special{color rgb 0 0 0}\allinethickness{0.254mm}\put(235,-25){\vector(0,-1){100}} 
        \special{color rgb 0 0 0}\allinethickness{0.254mm}\special{sh 0.99}\put(415,-25){\ellipse{4}{4}} 
        \special{color rgb 0 0 0}\put(405,-151){\shortstack{$\widehat O^A$}} 
        \special{color rgb 0 0 0}\allinethickness{0.254mm}\put(415,-25){\vector(1,-1){50}} 
        \special{color rgb 0 0 0}\allinethickness{0.254mm}\put(365,-75){\vector(1,1){50}} 
        \special{color rgb 0 0 0}\allinethickness{0.254mm}\put(365,-75){\vector(1,-1){50}} 
        \special{color rgb 0 0 0}\allinethickness{0.254mm}\put(465,-75){\vector(-1,-1){50}} 
        \special{color rgb 0 0 0}\allinethickness{0.254mm}\special{sh 0.99}\put(465,-75){\ellipse{4}{4}} 
        \special{color rgb 0 0 0}\allinethickness{0.254mm}\special{sh 0.99}\put(415,-125){\ellipse{4}{4}} 
        \special{color rgb 0 0 0}\allinethickness{0.254mm}\special{sh 0.99}\put(365,-75){\ellipse{4}{4}} 
        \special{color rgb 0 0 0}\put(355,-71){\shortstack{\scriptsize 1}} 
        \special{color rgb 0 0 0}\put(405,-21){\shortstack{\scriptsize 2}} 
        \special{color rgb 0 0 0}\allinethickness{0.254mm}\put(415,-25){\vector(0,-1){100}} 
        \special{color rgb 0 0 0} 
\end{picture}

%% file: misc.eepic
\setlength{\unitlength}{0.254mm}
\begin{picture}(680,146)(20,-173)
        \special{color rgb 0 0 0}\allinethickness{0.254mm}\special{sh 0.99}\put(285,-95){\ellipse{4}{4}} 
        \special{color rgb 0 0 0}\allinethickness{0.254mm}\special{sh 0.99}\put(335,-45){\ellipse{4}{4}} 
        \special{color rgb 0 0 0}\put(325,-171){\shortstack{$\vec P_3^A$}} 
        \special{color rgb 0 0 0}\allinethickness{0.254mm}\put(285,-95){\vector(1,1){50}} 
        \special{color rgb 0 0 0}\allinethickness{0.254mm}\special{sh 0.99}\put(385,-95){\ellipse{4}{4}} 
        \special{color rgb 0 0 0}\allinethickness{0.254mm}\special{sh 0.99}\put(80,-45){\ellipse{4}{4}} 
        \special{color rgb 0 0 0}\put(70,-171){\shortstack{$\vec K_{1,2}^A$}} 
        \special{color rgb 0 0 0}\allinethickness{0.254mm}\put(30,-95){\vector(1,1){50}} 
        \special{color rgb 0 0 0}\allinethickness{0.254mm}\put(30,-95){\vector(1,-1){50}} 
        \special{color rgb 0 0 0}\allinethickness{0.254mm}\special{sh 0.99}\put(80,-145){\ellipse{4}{4}} 
        \special{color rgb 0 0 0}\put(275,-91){\shortstack{\scriptsize 1}} 
        \special{color rgb 0 0 0}\put(325,-41){\shortstack{\scriptsize 2}} 
        \special{color rgb 0 0 0}\allinethickness{0.254mm}\special{sh 0.99}\put(30,-95){\ellipse{4}{4}} 
        \special{color rgb 0 0 0}\put(20,-91){\shortstack{\scriptsize 1}} 
        \special{color rgb 0 0 0}\put(70,-41){\shortstack{\scriptsize 2}} 
        \special{color rgb 0 0 0}\put(175,-171){\shortstack{$\vec K_{2,1}^N$}} 
        \special{color rgb 0 0 0}\allinethickness{0.254mm}\put(135,-95){\vector(1,-1){50}} 
        \special{color rgb 0 0 0}\allinethickness{0.254mm}\put(235,-95){\vector(-1,-1){50}} 
        \special{color rgb 0 0 0}\allinethickness{0.254mm}\special{sh 0.99}\put(235,-95){\ellipse{4}{4}} 
        \special{color rgb 0 0 0}\allinethickness{0.254mm}\special{sh 0.99}\put(185,-145){\ellipse{4}{4}} 
        \special{color rgb 0 0 0}\allinethickness{0.254mm}\special{sh 0.99}\put(135,-95){\ellipse{4}{4}} 
        \special{color rgb 0 0 0}\put(125,-91){\shortstack{\scriptsize 1}} 
        \special{color rgb 0 0 0}\put(240,-91){\shortstack{\scriptsize 2}} 
        \special{color rgb 0 0 0}\allinethickness{0.254mm}\put(335,-45){\vector(1,-1){50}} 
        \special{color rgb 0 0 0}\allinethickness{0.254mm}\special{sh 0.99}\put(440,-95){\ellipse{4}{4}} 
        \special{color rgb 0 0 0}\allinethickness{0.254mm}\special{sh 0.99}\put(490,-45){\ellipse{4}{4}} 
        \special{color rgb 0 0 0}\put(480,-171){\shortstack{$\vec P_3^N$}} 
        \special{color rgb 0 0 0}\allinethickness{0.254mm}\put(440,-95){\vector(1,1){50}} 
        \special{color rgb 0 0 0}\allinethickness{0.254mm}\special{sh 0.99}\put(540,-95){\ellipse{4}{4}} 
        \special{color rgb 0 0 0}\put(430,-91){\shortstack{\scriptsize 1}} 
        \special{color rgb 0 0 0}\put(545,-91){\shortstack{\scriptsize 2}} 
        \special{color rgb 0 0 0}\allinethickness{0.254mm}\put(490,-45){\vector(1,-1){50}} 
        \special{color rgb 0 0 0}\allinethickness{0.254mm}\special{sh 0.99}\put(640,-45){\ellipse{4}{4}} 
        \special{color rgb 0 0 0}\put(610,-171){\shortstack{$I^A$}} 
        \special{color rgb 0 0 0}\allinethickness{0.254mm}\put(590,-95){\vector(1,1){50}} 
        \special{color rgb 0 0 0}\allinethickness{0.254mm}\special{sh 0.99}\put(640,-145){\ellipse{4}{4}} 
        \special{color rgb 0 0 0}\allinethickness{0.254mm}\special{sh 0.99}\put(590,-95){\ellipse{4}{4}} 
        \special{color rgb 0 0 0}\put(580,-91){\shortstack{\scriptsize 1}} 
        \special{color rgb 0 0 0}\put(630,-41){\shortstack{\scriptsize 2}} 
        \special{color rgb 0 0 0}\allinethickness{0.254mm}\put(640,-145){\vector(0,1){100}} 
        \special{color rgb 0 0 0}\allinethickness{0.254mm}\put(640,-145){\vector(-1,1){50}} 
        \special{color rgb 0 0 0} 
\end{picture}